\documentclass[11pt]{amsart}
\usepackage{amssymb}
\usepackage[all]{xy}
\usepackage{arydshln}

\newcommand{\myendbibitem}{\relax}

\usepackage[linktocpage=true,hypertex]{hyperref}

\swapnumbers
\newtheorem{thm}[equation]{Theorem}
\newtheorem{prop}[equation]{Proposition}
\newtheorem{lem}[equation]{Lemma}
\newtheorem{cor}[equation]{Corollary}

\theoremstyle{definition}
\newtheorem{defn}[equation]{Definition}
\newtheorem{remark}[equation]{Remark}
\newtheorem{example}[equation]{Example}

\numberwithin{equation}{section}

\newcommand{\calI}{{\mathcal I}}
\newcommand{\calZ}{{\mathcal Z}}
\newcommand{\id}{\operatorname{id}}
\newcommand{\inv}{^{-1}}
\newcommand{\lra}{\longrightarrow}
\newcommand{\RMaps}{\operatorname{RMaps}}
\newcommand{\RMP}[1]{\RMaps_\PGLn(#1,\Mn)}
\newcommand{\Spec}{\operatorname{Spec}}
\newcommand{\Stab}{\operatorname{Stab}}
\newcommand{\tr}{\operatorname{tr}}

\newcommand{\Cmn}{C_{m,n}}
\newcommand{\Gmn}{G_{m,n}}
\newcommand{\Qmn}{Q_{m,n}}
\newcommand{\Tmn}{T_{m,n}}
\newcommand{\Umn}{U_{m,n}}

\newcommand{\M}{\operatorname{M}}
\newcommand{\Mn}{\M_n}
\newcommand{\Mnl}{(\Mn)^l}
\newcommand{\Mnm}{(\Mn)^m}
\newcommand{\GL}{{\operatorname{GL}}}
\newcommand{\PGLn}{{\operatorname{PGL}_n}}

\newcommand{\calC}{{\mathcal C}}
\newcommand{\calD}{{\mathcal D}}
\newcommand{\calE}{{\mathcal E}}

\newcommand{\Var}{\textit{Var}}
\newcommand{\PI}{\textit{PI}}
\newcommand{\Bir}{\textit{Bir}}
\newcommand{\CS}{\textit{CS}}

\newcommand{\knX}{k_n[X]}
\newcommand{\knY}{k_n[Y]}
\newcommand{\ad}{\operatorname{ad}}
\newcommand{\gl}{\operatorname{gl}}
\newcommand{\pgl}{\operatorname{pgl}}
\newcommand{\Lie}{\operatorname{Lie}}
\newcommand{\Max}{\operatorname{Max}}
\newcommand{\R}{\operatorname{R}}

\begin{document}

% To change the indentation in enumerate environments:
\setlength{\leftmargini}{2em}  % This must be done *after* \begin{document}

\title{Polynomial identity rings as rings of functions, II}

\date{{\bf Revision of November 12, 2011}}

\author{Nikolaus Vonessen}
\address{Department of Mathematical sciences, The University of Montana,
  Missoula, MT 59812-0864, USA}
\email{\href{mailto://nikolaus.vonessen@umontana.edu}{nikolaus.vonessen@umontana.edu}}
\urladdr{\href{http://www.math.umt.edu/vonessen}{http://www.math.umt.edu/vonessen}}
%\urladdr{\href{http://www.math.umt.edu/vonessen}{http://www.math.umt.edu/vonessen}\newpage\rm\tableofcontents}

\begin{abstract}
  In characteristic zero, Zinovy Reichstein and the author generalized
  the usual relationship between irreducible Zariski closed subsets of
  the affine space, their defining ideals, coordinate rings, and
  function fields, to a non-commutative setting, where ``varieties"
  carry a $\PGLn$-action, regular and rational ``functions" on them
  are matrix-valued, ``coordinate rings" are prime polynomial identity
  algebras, and ``function fields" are central simple algebras of
  degree~$n$.  In the present paper, much of this is extended to prime
  characteristic.  In addition, a mistake in the earlier paper is
  corrected.  One of the results is that the finitely generated prime
  PI-algebras of degree $n$ are precisely the rings that arise as
  ``coordinate rings" of ``$n$-varieties'' in this setting.  For $n =
  1$ the definitions and results reduce to those of classical affine
  algebraic geometry.
\end{abstract}



\maketitle

\section{Introduction}
\label{section:intro}

In characteristic zero, the usual relationship between irreducible
Zariski closed subsets of affine space, their defining ideals,
coordinate rings, and function fields was extended in \cite{nvar} to
the setting of PI-algebras.  The geometric objects are the {\it
  $n$-varieties}, certain $\PGLn$-invariant locally closed subsets of
$\Mnm$.  Here $\Mnm$ denotes the $m$-tuples of $n\times n$ matrices
over the algebraically closed base field $k$; $\PGLn$ acts on $\Mnm$
by simultaneous conjugation.  An irreducible $n$-variety $X$ has a
{\it PI-coordinate ring} $\knX$, which is a prime, finitely generated
PI-algebra over~$k$.  Moreover, up to isomorphism, every finitely
generated prime PI-algebra over $k$ arises in this way.  The total
ring of fractions of $\knX$ is a central simple algebra of degree~$n$,
called the {\it central simple algebra of rational functions on $X$}
and denoted by $k_n(X)$.

In~\cite{nvar}, we restricted attention to characteristic zero since
we were primarily interested in the ``rational'' theory (in
particular, in Theorem~1.2), where several proofs use the
characteristic zero assumption in an essential way.  But it is a
natural question to ask to what extent the results described above are
true in prime characteristic; in fact, in his MathSciNet
review~\cite{vaccarino}, Francesco Vaccarino termed it ``unfortunate''
that the positive characteristic case was not included.  The main
purpose of the present paper is to extend many of the results
in~\cite{nvar} to prime characteristic, some fully, others only
partially.  In particular, all results in Sections 3---5 extend fully
to prime characteristic (see Section~\ref{section:rem-to-sec-3-5}
below for more details and one minor exception).  The main results
of~\cite{nvar}, Theorems~1.1 and~1.2, require more discussion.  In
addition, we fix below an error affecting Theorem~1.1.  

Before continuing, we make several conventions.  This article should
be considered the second part of~\cite{nvar}, whose conventions and
notation we follow with one exception: the algebraically closed base
field $k$ is of arbitrary characteristic.  In this spirit, the numbers
of all results (and sections, etc.)  in the present paper are preceded
by ``II.'', while numbers without this prefix refer to items
in~\cite{nvar}.  Similarly, [1]---[30] refer to the references
of~\cite{nvar}, while the references of the present paper are numbered
[31] and higher.  These conventions should make it easier to
simultaneously read proofs in~\cite{nvar} together with the relevant
commentary in the present paper.

\begin{table}\small
\newcommand{\objparbox}[1]{\parbox{1.3in}{\raggedright
   \rule[2.5ex]{0mm}{0mm}#1\rule[-1.2ex]{0mm}{0mm}}}
\newcommand{\morparbox}[1]{\parbox{1.8in}{\raggedright 
   \rule[2.5ex]{0mm}{0mm}#1\rule[-1.2ex]{0mm}{0mm}}}
\newcommand{\thmparbox}[1]{\parbox{.7in}{\centering
   \rule[2.5ex]{0mm}{0mm}#1\rule[-1.2ex]{0mm}{0mm}}}
\begin{tabular}{|c|l|l|c|}\hline
\em Category & \em Objects & \em Morphisms
 & \multicolumn{1}{|l|}{\thmparbox{\em Relevant Theorems}}\\ \hline
$\PI_n$ & \objparbox{finitely generated\\ prime $k$-algebras of PI-degree~$n$}
  & \morparbox{$k$-algebra
  homomorphisms $\alpha\colon R\to S$ such that for any $k$-algebra
  surjection $\phi\colon S\to\Mn$, $\phi\circ\alpha$ is also
  surjective}
  & \ref{thm1prime}\\ \hline
$\CS_n$
  & \objparbox{central simple algebras of degree $n$ whose center is a
    finitely generated field extension of $k$}
  & \morparbox{$k$-algebra homomorphisms (necessarily injective)}
  & \parbox{.7in}{\centering \ref{thm:B} \\
    \ref{thm:Dnew} \\ 1.2 (char $0$)}
    \\ \hline
\end{tabular}

\medskip
\caption{\label{table:alg-cat}\em The algebraic categories}
\end{table}

Our main results involve several algebraic and geometric categories,
see Tables~\ref{table:alg-cat} and~\ref{table:geom-cat}.  The
morphisms in the category $\PI_n$ were incorrectly defined
in~\cite{nvar}; an alternate characterization of these morphisms is
given in Remark~\ref{rem:alpha-in-PIn}.

\begin{table}\small
\newcommand{\objparbox}[1]{\parbox{1.75in}{\raggedright
   \rule[2.5ex]{0mm}{0mm}#1\rule[-1.2ex]{0mm}{0mm}}}
\newcommand{\morparbox}[1]{\parbox{1.55in}{\raggedright 
   \rule[2.5ex]{0mm}{0mm}#1\rule[-1.2ex]{0mm}{0mm}}}
\newcommand{\thmparbox}[1]{\parbox{.55in}{\centering
   \rule[2.5ex]{0mm}{0mm}#1\rule[-1.2ex]{0mm}{0mm}}}
\begin{tabular}{|c|l|l|c|}\hline
\em Category & \em Objects & \em Morphisms 
 & \multicolumn{1}{|l|}{\thmparbox{\raggedright \em\small Relevant Theorem}}\\ \hline
$\Var_n$ & 
  & \morparbox{regular maps of $n$-varieties (Def.~6.1(a))} 
  & \ref{thm1prime}\\ \cline{1-1}\cdashline{2-2}\cline{3-4} %\hline
$\calC_n$
  & \smash{$\left.\rule[-8ex]{0ex}{0ex}\right\}$}
    irreducible $n$-varieties
  & \morparbox{dominant rational maps of $n$-varieties
    (Def.~\ref{defn:7.5new})}
  & \ref{thm:B}\\ \cline{1-1}\cdashline{2-2}\cline{3-4} %\hline
$\calD_n$
  & 
  & \morparbox{\ \\ \ }
    
  & \\ \cline{1-2}\cdashline{3-3}\cline{4-4} %\hline
$\calE_n$
  & \objparbox{irreducible $\PGLn$-varieties that are
     $\PGLn$-equivariantly birationally isomorphic to an $n$-variety} 
  & \smash{$\left.\rule[-11ex]{0ex}{0ex}\right\}$}
    
    \parbox{1.2in}{\raggedright$\PGLn$-equivariant dominant rational maps}
  & \ref{thm:Dnew}\\ \cline{1-2}\cdashline{3-3}\cline{4-4} %\hline
$\Bir_n$ 
  & \objparbox{irreducible generically free $\PGLn$-varieties} 
  & 
  & \thmparbox{1.2 \hbox{(char $0$)}}\\ \hline
\end{tabular}

\medskip
\caption{\label{table:geom-cat}\em The geometric categories}
\end{table}

Each of the geometric categories, except for $\Var_n$, is a
subcategory of the one listed below it.  That $\calE_n$ is a
subcategory of $\Bir_n$ follows from Lemma~\ref{lem:n-var:gen-free}.
Note that the inclusion of $\calD_n$ into $\calE_n$ is a category
equivalence.  In characteristic zero, $\calC_n=\calD_n$ and
$\calE_n=\Bir_n$; this follows from Proposition~7.3 and Lemma~8.1,
which we will discuss below.  It is not known if these equalities also
hold in prime characteristic (cf.\ the open questions listed at the
end of the introduction).

\begin{thm} \label{thm1prime}
  The functor defined by
  \[ \begin{array}{rcl}  X &\mapsto &k_n[X] \\
     (f \colon X \lra Y) &\mapsto &(f^* \colon k_n[Y] \lra k_n[X])
     \end{array} \]%
  is a contravariant equivalence of categories between $\Var_n$ and
  $\PI_n$. 
\end{thm}

Here, as in the following theorems, $f^*(g)=g\circ f$ (as maps
$X\to\Mn$).

Theorem~\ref{thm1prime} has the same statement as Theorem~1.1; but the
latter is false without the correct definition of the morphisms in the
category $\PI_n$ given in Table~\ref{table:alg-cat}.  Beyond
Theorem~1.1, the only results in~\cite{nvar} that need to be corrected
are Remark~6.2 and Lemma~6.3, see Section~\ref{section:rem-sec-6}.

In characteristic zero, every irreducible generically free
$\PGLn$-variety is birationally isomorphic (as a $\PGLn$-variety) to
an irreducible $n$-variety (Lem\-ma~8.1).  Using this, we established
a category equivalence from the category~$\Bir_n$ to the
category~$\CS_n$ (Theorem~1.2).  To prove this result in prime
characteristic would require answering several open questions in the
affirmative (see below).  Restricting attention to certain
subcategories of $\Bir_n$ we are able to obtain two results in the
spirit of Theorem~1.2, the first of which is proved in
Section~\ref{section:first-extension-S7}.

\begin{thm}\label{thm:B}
  The functor defined by
  \[ \begin{array}{rcl}      X &\mapsto &k_n(X) \\
    (f \colon X \dasharrow Y) &\mapsto &(f^* \colon k_n(Y)
    \hookrightarrow k_n(X)) \end{array} \]%
  is a contravariant category equivalence from the category~$\calC_n$
  to the category~$\CS_n$. 
\end{thm}

Note one particular fact contained in this result: Every central
simple algebra in $\CS_n$ is isomorphic to $k_n(X)$ for some
irreducible $n$-variety $X$ (unique up to birational isomorphism of
$n$-varieties).  This is Theorem~7.8, which remains true in prime
characteristic. 

For an irreducible $\PGLn$-variety $X$, $\RMaps_\PGLn(X,\Mn)$ denotes
the $k$-algebra of $\PGLn$-equivariant rational maps $X\dasharrow\Mn$.
In characteristic zero, if $X$ is a generically free $\PGLn$-variety,
$\RMaps_\PGLn(X,\Mn)$ is a central simple algebra of degree~$n$, and
if $X$ is an irreducible $n$-variety, the natural inclusion of
$k_n(X)$ into $\RMaps_\PGLn(X,\Mn)$ is an isomorphism
(Proposition~7.3).  It is not known if these facts remain true in
prime characteristic.  We can prove that for an irreducible
$n$-variety $X$, $\RMaps_\PGLn(X,\Mn)$ is a central simple algebra of
degree~$n$ whose center is a finite, purely inseparable extension of
the center of $k_n(X)$ (Proposition~\ref{prop:A}).  This enables us to
prove the following result in
Section~\ref{section:second-extension-S7}.

\begin{thm}\label{thm:Dnew}
  The functor defined by
  \[ \begin{array}{rcl}      X &\mapsto &\RMaps_\PGLn(X,\Mn) \\
    (f \colon X \dasharrow Y) &\mapsto &\bigl(f^\star \colon
    \RMaps_\PGLn(Y,\Mn) \hookrightarrow
    \RMaps_\PGLn(X,\Mn)\bigr) \end{array} \]%
  is a full and faithful contravariant functor from the
  category~$\calE_n$ to the category~$\CS_n$.
\end{thm}

In Section~\ref{section:RMaps}, we study several basic questions
regarding the algebra $\RMaps_\PGLn(X,\Mn)$.  If $X$ is an irreducible
$n$-variety, $k_n(X)$ embeds into $\RMaps_\PGLn(X,\Mn)$, and the
latter is a central simple algebra of degree~$n$, see
Proposition~\ref{prop:A}.  Hence by Theorem~7.8, there is an
irreducible $n$-variety $Y$ such that $\RMaps_\PGLn(X,\Mn)\cong
k_n(Y)$.  We show that the natural embedding of $k_n(Y)$ into
$\RMaps_\PGLn(Y,\Mn)$ is an isomorphism (even if the same should not
be true for $X$), that $Y$ is unique up to birational isomorphism of
$n$-varieties, and that $Y$ is birationally equivalent to $X$ as
$\PGLn$-varieties (though maybe not as $n$-varieties), see
Corollary~\ref{cor:9.7} and Proposition~\ref{prop:9.3}.

We also address the question for which irreducible $\PGLn$-varieties
$X$, the algebra $\RMaps_\PGLn(X,\Mn)$ is a central simple algebra of
degree~$n$.  In characteristic zero, a necessary and sufficient
condition is that the $\PGLn$-action on $X$ is generically free, see
\cite[Lemma~2.8]{gacsa}.  In prime characteristic, we give a different
criterion in Corollary~\ref{cor:9.3}: the existence of a
$\PGLn$-equivariant dominant rational map $X\dasharrow Y$ for some
irreducible $n$-variety $Y$.

We conclude with several interrelated open questions in prime
characteristic.  Let $X$ be an irreducible generically free
$\PGLn$-variety, and set $\R(X)=\RMaps_\PGLn(X,\Mn)$.
\begin{enumerate}
\item Is $X$ birationally isomorphic {\upshape(}as
  $\PGLn$-variety{\upshape)} to an $n$-variety? Equivalently, is
  $\calE_n=\Bir_n$? (Cf.\ Lemma~8.1.)
\item Is $\R(X)$ a central simple algebra of degree~$n$? (Cf.\
  \cite[Lemma~2.8]{gacsa}, Proposition~\ref{prop:A}, and
  Corollary~\ref{cor:9.3}.) 
\item If $X$ is an $n$-variety, is $k_n(X)=\R(X)$?  That is, is the
  center of $k_n(X)$ always $k(X)^\PGLn$?  Equivalently, is
  $\calC_n=\calD_n$?  (Cf.\ Propositions~7.3 and~\ref{prop:A}.)
\item Is every central simple algebra in $\CS_n$ isomorphic to $\R(X)$
  for some $X$ (or for some irreducible $n$-variety $X$)? (Cf.\
  Theorem~7.8.)
\end{enumerate}

\smallskip

\noindent %
{\bf Acknowledgments.} The author is grateful to Zinovy Reichstein and
the referee for helpful comments and suggestions.

\section{Invariant-Theoretic Background for Prime Characteristic}
\label{section:inv-background}

Donkin \cite{donkin} proved Procesi's Conjecture
\cite[p.~308/309]{procesi1}: The ring of invariants 
\[\Cmn=k[\Mnm]^\PGLn\]
for the action of $\PGLn$ on $\Mnm$ by simultaneous conjugation is the
algebra generated by the elements
\[c_s(X_{i_1}X_{i_2}\cdots X_{i_r})\,,\]
where the $X_i$ are generic matrices and $c_s$ denotes the coefficient
in degree $n-s$ of the characteristic polynomial.  Note that Donkin
describes the generators as the functions
\[(x_1,\ldots,x_m)\mapsto\tr\bigl(x_{i_1}x_{i_2}\cdots
x_{i_r},\bigwedge{}^s(k^n)\bigr)\,.\]
But it is easy to see that this function is the function $\Mnm\to k$
induced by $c_s(X_{i_1}X_{i_2}\cdots X_{i_r})$ by verifying that
$\tr\bigl(A,\bigwedge{}^s(k^n)\bigr)=c_i(A)$ for every $A\in \Mn$.
(Prove this first for diagonalizable $A$ and then use a density
argument.)

Other proofs of Procesi's Conjecture are given in \cite{zubkov:MR96c},
\cite{domokos-zubkov} and \cite{domokos}.  We remark that the ring of
invariants $\Cmn$ has been studied extensively; for work in prime
characteristic see, e.g., \cite{zubkov:MR95c}, \cite{zubkov:MR2002h},
and \cite{domokos-kuzmin-zubkov}.

We denote by $\Qmn$ the affine variety with coordinate ring $\Cmn$.
Since $\PGLn$ is reductive, the latter ring is affine, and the natural
map
\begin{equation}
  \label{eq:pi}
  \pi\colon \Mnm\to\Qmn  
\end{equation}
is a categorical quotient map.  In particular, $\pi$ is surjective,
$\pi$ maps closed $\PGLn$-stable subsets to closed subsets, and $\Cmn$
separates disjoint closed $\PGLn$-stable subsets of $\Mnm$.  These are
all well-known facts, see \cite[Section 13.2, and in particular
Theorem 13.2.4]{FSR} or \cite[Theorem 6.1 on page~97 (see the
definition on page~94)]{dolgachev}.  Separation of disjoint closed
invariant subsets is also proved in \cite[Corollary~A.1.3, p.~151]{mf}.

As usual, $\Tmn$ is the trace ring of the generic matrix ring
\[\Gmn=k\{X_1,\ldots, X_m\}\]
(which is generated by the $m$ generic $n\times n$ matrices
$X_1$,\ldots, $X_m$).  In prime characteristic, $\Tmn$ is generated
over $\Gmn$ by adjoining all coefficients of the characteristic
polynomial of each element of $\Gmn$.  So $\Tmn$ is a central
extension of $\Gmn$.  As is well-known, $\Tmn$ is affine and finite
over its center.

We can think of the elements of $\Tmn$ as $\PGLn$-equivariant regular
functions $\Mnm\to\Mn$.  This allows us to identify $\Cmn$ with the
center of $\Tmn$, as we did in characteristic zero: As usual, $\Tmn$
is a subalgebra of $n\times n$ matrices over a large polynomial ring.
Since $\Tmn$ has PI-degree $n$, its center is contained in the center
of that matrix ring, so consists of scalar matrices.  The center of
$\Tmn$ thus consists of $\PGLn$-equivariant scalar-valued functions,
i.e., $\PGLn$-invariant functions $\Mnm\to kI_n$, where $I_n$ denotes
the identity matrix in $\Mn$.  Identifying $k$ with $k I_n$, we see
that the center of $\Tmn$ embeds into $\Cmn$.  But by Donkin's result,
the center of $\Tmn$ contains the generators of $\Cmn$, so that the
embedding of the center of $\Tmn$ into $\Cmn$ is actually an
isomorphism.  

In characteristic zero, Procesi proved that $\Tmn$ consists precisely
of the regular $\PGLn$-equivariant maps $\Mnm\to\Mn$, see
\cite[Theorem~2.1]{procesi1}.  This is also true in prime
characteristic.  Indeed, using Donkin's result, Procesi's proof goes
through nearly literally: one can still use the nondegeneracy of the
trace form.  But $\bar f$ is now a polynomial in Donkin's generators.
Note that $\bar f$ is ``homogeneous in $X_{i+1}$ of degree 1''.  Now
$c_s$ applied to a monomial $M$ in $X_1,\ldots,X_{i+1}$ involving $r$
copies of $X_{i+1}$ is homogeneous of degree $rs$.  Since $c_1=\tr$,
one can still write $\bar f$ exactly as in the second displayed
equation of Procesi's proof, and the rest of the proof goes through.
(To make the degree argument precise: Recall that $\Gmn$ and $\Tmn$
are contained in the $n\times n$ matrices over the commutative
polynomial ring in the entries of the generic matrices.  Define a
degree function by giving each entry of $X_{i+1}$ degree~$1$ and all
other variables degree~$0$.  Then $\bar f$ is homogeneous of
degree~$1$, and $c_s(M)$ is homogeneous of degree~$rs$.)

\section{Preliminaries on Trace Rings}
\label{section:tracering-background}

If $R$ is a prime PI-ring of PI-degree $n$, its total ring of
fractions is a central simple algebra of degree~$n$, and the trace
ring $T(R)$ of $R$ is obtained from $R$ by adjoining to $R$ all
coefficients of the reduced characteristic polynomial of every element
of $R$.

We will repeatedly use the well-known fact, due to Amitsur, that
\begin{equation}
  \label{eq:amitsur}
  T(T(R)) = T(R)\,.  
\end{equation}
This follows immediately from~\cite{amitsur:multilin}: Denote by $Z$
the center of $T(R)$.  Then every element of $T(R)$ is a $Z$-linear
combination of elements of $R$.  By~\cite{amitsur:multilin}, the
coefficients of the characteristic polynomial of such a linear
combination are polynomial expressions (with coefficients in $Z$) in
the coefficients of the characteristic polynomials of certain elements
of $R$, so belong to $T(R)$, implying~\eqref{eq:amitsur}.

We will also need the following result:

\begin{thm} {\sc(Amitsur-Small)}\label{thm:amitsur-small}\ \ 
  Let $\phi\colon R\to S$ be a surjective ring homomorphism between
  prime PI-rings of the same PI-degree~$n$.  
  Then $\phi$ extends uniquely to a homomorphism $\phi'\colon T(R)\to
  T(S)$ that preserves characteristic polynomials; i.e.,
  $\phi'(c_s(r))=c_s(\phi(r))$ for all $r\in R$ and $1\leq s\leq n$.
  It follows immediately that also $\phi'$ is surjective.  Moreover,
  if $R=T(R)$ {\upshape(}e.g., if $R=\Tmn${\upshape)}, then $S=T(S)$
  and $\phi$ preserves characteristic polynomials. \qed
\end{thm}

This is \cite[Theorem~2.2]{amitsur:small}.  (Note the typo in the
statement of this result: ``into'' should be ``onto''.  Note also that
the trace ring of $R$ is in~\cite{amitsur:small} denoted by $T(R)R$.)

\section{Remarks to Section 2}
\label{rem:section2}

Everything in Section~2 extends to prime characteristic if one uses
the facts and the modified definitions from
Sections~\ref{section:inv-background}
and~\ref{section:tracering-background}.  Note that the statement of
Lemma 2.6, though correct in prime characteristic, can be strengthened
in the natural way, see Lemma~\ref{lem:2.6-in-prime-char}.

{\bf 2.1} was addressed in Section~\ref{section:inv-background}.
Note that $\Cmn$ has different generators in prime characteristic.

{\bf Proposition 2.3} is also true in prime characteristic; since it
is of fundamental importance for this paper, we restate it formally.
We denote by $\Umn$ the following $\PGLn$-stable dense open subset of
$\Mnm$:
\[ \Umn = \{ (a_1,\ldots,a_m) \in \Mnm \, | \,
   \text{$a_1,\ldots,a_m$ generate $\Mn$ as $k$-algebra}\} \]
Note that
\begin{equation}
  \label{eq:n-var-triv-stab}
  \Stab_\PGLn(x)=1 \qquad   \forall\,\, x\in\Umn\,\,.
\end{equation}
Here $\Stab_\PGLn(x)$ denotes the group-theoretic stabilizer of $x$ in
$\PGLn$.  Recall the categorical quotient map $\pi\colon\Mnm\to\Qmn$
from~\eqref{eq:pi}.

\begin{prop} \label{prop:2.3-in-prime-char}
  Proposition~2.3 is true in arbitrary characteristic.  That is:
  \begin{enumerate}
  \item[\upshape(a)] If $x \in U_{m, n}$ then $\pi^{-1}(\pi(x))$ is
    the $\PGLn$-orbit of $x$.
  \item[\upshape(b)]$\PGLn$-orbits in $\Umn$ are closed in $\Mnm$.
  \item[\upshape(c)]$\pi$ maps closed $\PGLn$-invariant sets in $\Mnm$
    to closed sets in $Q_{m, n}$.
  \item[\upshape(d)]$\pi(U_{m,n})$ is Zariski open in $\Qmn$.
  \item[\upshape(e)]If $Y$ is a closed irreducible subvariety of
    $Q_{m, n}$ then $\pi^{-1}(Y) \cap U_{m, n}$ is irreducible in
    $\Mnm$.
  \end{enumerate}
\end{prop}

\begin{proof}
  Given (a) (which we prove below), (b) follows immediately.  Part~(c)
  was addressed in Section~\ref{section:inv-background}.
  Concerning~(d) and~(e): the proofs of the corresponding parts of
  Proposition~2.3 go through without changes.

  (a) This is an adaptation of part of the proof of
  \cite[(12.6)]{artin1}.  Let $x\in \Umn$, and let $\phi=\phi_x$ be
  the corresponding surjective (thus simple) representation
  $\Gmn\to\Mn$, $p\mapsto p(x)$.  Let $x'\in\pi\inv(\pi(x))$, and let
  $\phi'=\phi_{x'}\colon \Gmn\to\Mn$ be the corresponding (potentially
  not surjective) representation.  We have to show that $x'$ belongs
  to the $\PGLn$-orbit of $x$.  To do this, we will twice replace $x'$
  by another element in its own $\PGLn$-orbit, and then prove $x'=x$.

  Note that the representations $\phi$ and $\phi'$ preserve
  characteristic polynomials.  (For $\phi'$, this follows, e.g., from
  the fact that, using the notation in~2.4, $\phi'$ can be extended to
  the homomorphism $\Mn(k[x_{i,j}^{(h)}])\to\Mn$ given by evaluating
  the indeterminates $x_{i,j}^{(h)}$ at the corresponding entries of
  the $m$-tuple $x'\in\Mnm$.)  Since $\pi(x)=\pi(x')$, $f(x)=f(x')$
  for all $f\in\Cmn$.  Let $p\in\Gmn$, and denote by $c_s(p)$ the
  coefficient in degree $n-s$ of the characteristic polynomial of $p$.
  Since $c_s(p)\in\Cmn$,
  \[c_s(\phi(p))=\phi(c_s(p))=c_s(p)(x)=c_s(p)(x')
  =\phi'(c_s(p))=c_s(\phi'(p))\,.\]
  Consequently, for all $p\in\Gmn$, $\phi(p)$ and $\phi'(p)$ have the
  same trace and the same characteristic polynomial; we will use this
  repeatedly below.

  Since $\phi$ is onto, there is a $z\in\Gmn$ such that $\phi(z)$ is a
  diagonal matrix with distinct eigenvalues.  Since $\phi'(z)$ has the
  same characteristic polynomial as $\phi(z)$, it is in the
  $\PGLn$-orbit of $\phi(z)$ in $\Mn$.  Replacing $x'$ by some other
  element in its $\PGLn$-orbit (and, of course, also changing
  $\phi'=\phi_{x'}$), we may assume that $\phi(z)=\phi'(z)$ is a
  diagonal matrix with distinct eigenvalues.  It is possible to
  express the elementary matrix units $e_{i,i}$ as polynomials (with
  coefficients in $k$) in this diagonal matrix.  Hence there exist
  $z_{i,i}\in\Gmn$ such that $\phi(z_{i,i})=\phi'(z_{i,i})=e_{i,i}$
  for all $i=1,\ldots,n$.

  Since $\phi$ is onto, we can choose, for $i,j\in\{1,\ldots,n\}$ with
  $i\neq j$, elements $\tilde z_{i,j}$ in $\Gmn$ such that
  $\phi(\tilde z_{i,j})=e_{i,j}$.  For $i\neq j$, set
  $z_{i,j}=z_{i,i}\tilde z_{i,j}z_{j,j}$.  Then
  $\phi(z_{i,j})=e_{i,j}$ for all $i$ and $j$.  Now for $i\neq j$,
  $\phi'(z_{i,j})=e_{i,i}\phi'(\tilde
  z_{i,j})e_{j,j}=\alpha_{i,j}e_{i,j}$ for certain scalars
  $\alpha_{i,j}\in k$.  Now $\phi(z_{i,j}z_{j,i})=e_{i,i}$ and
  $\phi'(z_{i,j}z_{j,i})=\alpha_{i,j}\alpha_{j,i}e_{i,i}$ have the
  same trace, i.e., $1=\alpha_{i,j}\alpha_{j,i}$.  In particular,
  $\alpha_{i,j}\neq0$.  Hence
  $\gamma=e_{1,1}+\alpha_{1,2}e_{2,2}+\cdots+\alpha_{1,n}e_{n,n}$ is
  in $\GL_n$.  

  Set $\phi''=\gamma\phi'\gamma\inv$.  Then one checks readily that
  $\phi''(z_{i,i})=e_{i,i}$ for $i=1,\ldots, n$, and that
  $\phi''(z_{1,j})=e_{1,j}$ for $j=2,\ldots, n$.  Replacing $x'$ by
  $\gamma x'\gamma\inv$ (and thus $\phi'$ by
  $\phi''=\gamma\phi'\gamma\inv$), we may assume that $\alpha_{1,j}=1$
  for $j=2,\ldots,n$.  Since $\alpha_{i,j}\alpha_{j,i}=1$, also
  $\alpha_{j,1}=1$.  So for all $i$ and $j$,
  \[\phi'(z_{1,i})=e_{1,i} \quad\text{and}\quad \phi'(z_{j,1})=e_{j,1}\,.\]

  Finally, let $p\in\Gmn$ be arbitrary.  Say
  $\phi(p)=(\beta_{i,j}), \phi'(p)=(\beta'_{i,j})\in\Mn$.  Then for
  all $i$ and $j$,
  \begin{align*}
    \beta'_{i,j}&=\tr(\beta'_{i,j}e_{1,1})=\tr(e_{1,i}\phi'(p)e_{j,1})
               =\tr(\phi'(z_{1,i}pz_{j,1}))\\
    &=\tr(\phi(z_{1,i}pz_{j,1}))=\tr(e_{1,i}\phi(p)e_{j,1})
     =\tr(\beta_{i,j}e_{1,1})=\beta_{i,j}\,,
  \end{align*}
  so that $\phi'(p)=\phi(p)$.  Thus $\phi'=\phi$ and $x'=x$,
  concluding the proof of~(a).
\end{proof}

{\bf 2.4} was addressed in Section~\ref{section:inv-background}.

{\bf 2.5---2.10} carry over to prime characteristic.  We make two remarks:

{\bf Lemma 2.6} remains true in prime characteristic as stated, but
parts~(a) and~(b) can be strengthened as follows:

\begin{lem}\label{lem:2.6-in-prime-char}
  Let $J$ be a prime ideal in $\Spec_n(\Tmn)$, i.e., a prime ideal of
  $\Tmn$ such that $\Tmn/J$ has PI-degree~$n$.
  \begin{enumerate}
  \item[\upshape(a)]The natural projection $\phi \colon T_{m, n} \lra
    T_{m, n}/J$ preserves characteristic polynomials, and $\Tmn/J$ is
    the trace ring of $\Gmn/(J \cap G_{m, n})$.
  \item[\upshape(b)]For every $p\in J$, $c_i(p)\in J$ for
    $i=1,\ldots,n$.  In particular, $\tr(p) \in J$.
  \end{enumerate}
\end{lem}

\begin{proof}
  (a) Since $\Tmn$ is its own trace ring by \eqref{eq:amitsur},
  Theorem~\ref{thm:amitsur-small} first implies that $\phi$ preserves
  characteristic polynomials, and then, applying it to the natural map
  $\Gmn\to\Gmn/(J\cap\Gmn)$, that $\Tmn/J$ is the trace ring of
  $\Gmn/(J \cap G_{m, n})$.  (b)~Since
  $\phi(c_i(p))=c_i(\phi(p))=c_i(0)=0$, $c_i(p)\in J$.
\end{proof}

{\bf Lemma 2.10:} The use of \cite[Theorem~(9.2)]{artin1} can be
avoided with a short, direct argument: If $\phi_i$ and $\phi_j$ had
the same kernel, say $I$, then they would induce $k$-algebra
isomorphisms $\Gmn/I\to\Mn$.  It follows easily that $\phi_i$ and
$\phi_j$ would then differ by an automorphism of $\Mn$, i.e., by an
element of $\PGLn$, a contradiction.

\section{Remarks to Sections 3--5}
\label{section:rem-to-sec-3-5}

Using the facts and the modified definitions from
Sections~\ref{section:inv-background}
and~\ref{section:tracering-background}, and the comments in
Section~\ref{rem:section2} regarding Section~2, the content of
Sections 3---5 (definitions, results, and proofs) remain true in prime
characteristic, with the possible exception of Remark~3.2.  In
particular, the following results remain true in prime characteristic:
the weak and the strong Nullstellensatz for prime ideals
(Propositions~5.1 and~5.3), and the inclusion-reversing bijections
between the irreducible $n$-varieties in $\Umn$ and $\Spec_n(\Gmn)$
(resp., $\Spec_n(\Tmn)$) (Theorem~5.7).

Only a few items require further comment.  We begin with a lemma which
is likely known.

\begin{lem}\label{lem:n-var:gen-free}
  Also in prime characteristic, $n$-varieties are generically free as
  $\PGLn$-varieties.
\end{lem}

In characteristic zero, the action of a linear algebraic group $G$ on
a $G$-variety $X$ is called {\em generically free} if there is a dense
open subset $U$ of $X$ such that for any $x\in U$ the group-theoretic
stabilizer $G_x$ is trivial.  If $G=\PGLn$ and $X$ is an $n$-variety
in $\Umn$, then $G_x$ is trivial for all $x\in X$, in arbitrary
characteristic, see \eqref{eq:n-var-triv-stab}.  So the $\PGLn$-action
on an $n$-variety is generically free in characteristic zero.  In
prime characteristic, one has to check an additional property in order
to ensure that the $\PGLn$-action on the $n$-variety $X$ is
generically free (see \cite[Section 4]{berhuy:favi} for the precise
definition and relevant results).  We will not be using
Lemma~\ref{lem:n-var:gen-free} in the sequel, except to note that the
category ${\mathcal E}_n$ is a subcategory of $\Bir_n$, as already
remarked in the introduction.  Therefore we will only sketch the
proof. We are grateful to Zinovy Reichstein for showing it to us.

\begin{proof}[Outline of the Proof]
  By \cite[Lemma 4.2($i'$)]{berhuy:favi}, it suffices to show that for
  any closed point $x\in X$, the Lie algebra of the scheme-theoretic
  stabilizer, which we will denote by $\tilde G_x$, is trivial.  We
  may assume $x\in X(k)\subseteq\Mnm$ (if not, replace $k$ by the
  algebraic closure of the function field of the irreducible
  subvariety of $X$ corresponding to~$x$).  If $f\colon G \to X$ is
  given by $f(g)=g\cdot x=gxg^{-1}$, then $\Lie(\tilde G_x)$ is the
  kernel of the derivative $df_1$.

  The derivative of the conjugation representation of $\GL_n$ on
  $\M_n=\gl_n$ is the adjoint representation of $\gl_n$ on itself,
  which takes $A \in \gl_n$ to $\ad(A)\colon X \mapsto [A, X]$.  Since
  the center of $\GL_n$ acts trivially on $\Mn$, the conjugation
  action of $\GL_n$ on $\Mn$ descends to $\PGLn$, and the Lie
  derivative for this $\PGLn$-action is still $\ad(A)\colon x \mapsto
  [A, x]$, except that we are now viewing $A$ as an element of $\pgl_n
  = \gl_n/kI_n$.  Similarly, for the action of $\PGLn$ on $\Mnm$ by
  simultaneous conjugation, the Lie action of $A\in\pgl_n$ on $\Mnm$
  is given by $x=(x_1,\ldots,x_m) \mapsto ([A, x_1],\ldots, [A,
  x_m])$.  For a fixed $x\in\Umn$ the kernel of this map is
  $\Lie(\tilde G_x)$; it is trivial, since the only matrices in $\Mn$
  commuting with all $x_i$ are the scalar matrices.  So also in prime
  characteristic, $n$-varieties are generically free as
  $\PGLn$-varieties.
\end{proof}

{\bf Remark 3.2:} The first two sentences are true in prime
characteristic.  The last may not be --- the status of Proposition~7.3
is not known in prime characteristic (see the introduction to
Section~\ref{section:first-extension-S7} and
Proposition~\ref{prop:A}).

{\bf Example 3.3} works in characteristic $\neq 2$ since the proof of
\cite[Theorem~2.9]{friedland} goes through under this assumption,
showing that $U_{2,2}$ is defined by the non-vanishing of
$c=c(X_1,X_2)$.  (Note the typo in the definition of $c(X_1,X_2)$: the
lone ``det'' should be ``tr''.)  Now if $f$ were in $T_{2,2}$, then
$c$ would be an invertible element of $T_{2,2}$.  Hence for any
$a\in(\M_2)^2$, $\phi_a(c)=c(a)\neq0$, contradicting that $c(a)=0$ for
$a\not\in U_{2,2}$.

{\bf Remark 3.4:} This goes through.  The proof in part (b) uses the
following: If $\phi$ is a surjective $k$-algebra homomorphism
$\Tmn\to\Mn$, then
\begin{equation}
  \label{eq:remark3.4}
  \phi(p(X_1,\ldots,X_m))=p(\phi(X_1),\ldots,\phi(X_m))
\end{equation}
for all $p\in\Tmn$.
It suffices to verify this for the $k$-algebra generators of $\Tmn$.
It is clear for $p\in\Gmn$.  For $p=c_s(h)$ with $h\in\Gmn$ it follows
from the results summarized in
Section~\ref{section:tracering-background}: By~\eqref{eq:amitsur},
$T(\Tmn)=\Tmn$, and clearly $T(\Mn)=\Mn$.  Hence by
Theorem~\ref{thm:amitsur-small}, $\phi$ preserves characteristic
polynomials, so that \eqref{eq:remark3.4} holds for $p=c_s(h)$ with
$h\in\Gmn$.

{\bf Lemma 3.6:} The proof of (a) uses \cite[Theorem~1]{schelter:78},
which does not need characteristic zero.  In (b), as in many other
instances in later sections, one has to use results summarized in
Section~\ref{section:inv-background}.  In this particular instance one
needs that also in prime characteristic, $C$ and $\bar X$ can be
separated by a function $f\in \Cmn$, and that $f$ can be thought of as
a central element of $\Tmn$.

{\bf Corollary 5.2:} Note that the proof shows that for every
irreducible $n$-variety~$X$, $\calI_T(X)$ is a prime ideal of $\Tmn$
of PI-degree~$n$.  The same argument shows that $\calI(X)$ is a prime
ideal of $\Gmn$ of PI-degree~$n$.

\section{Remarks to Section 6}
\label{section:rem-sec-6}

As mentioned in the introduction, Theorem~1.1 is incorrect as stated,
but it can be easily corrected, and the corrected version,
Theorem~\ref{thm1prime}, is also valid in prime characteristic.  The
only other results in~\cite{nvar} that need to be corrected are
Remark~6.2 and Lemma~6.3.  With these changes, all
results of Section~6 are also valid in prime characteristic.

The problem stems from the definition of $\alpha_*$ in
Definition~6.1(d).  Let $X \subset U_{m, n}$ and $Y \subset U_{l, n}$
be $n$-varieties, and let
$\alpha\colon \knY\to \knX$
be a $k$-algebra homomorphism.  Denote the generic matrices generating
$G_{l,n}$ by $X_1$, \ldots, $X_l$, and their images in $\knY$ by
$\overline{X_1}$, \ldots, $\overline{X_l}$.  Set
$f_i=\alpha(\overline{X_i})$.  We define a regular map
\[\alpha_\#\colon X\to\Mnl\]
by $\alpha_\#(x)=(f_1(x),\ldots,f_l(x))$.  If $\alpha_\#(X)\subseteq
Y$, then $\alpha_*(x)=\alpha_\#(x)$ defines a regular map of
$n$-varieties 
\[\alpha_*\colon X\to Y\,.\]
But $\alpha_\#(X)\subseteq Y$ may not be true.  One easily checks that
the image of $\alpha_\#$ is contained in the set~$\calZ_0$ of zeroes
of $\calI(Y)$ in $(\Mn)^l$ (cf.~\eqref{lem:first-ext-of-def-7.5:eqn}),
but it is not necessarily contained in $U_{l,n}$, and thus not
necessarily contained in $Y=\calZ(\calI(Y))=\calZ_0\cap U_{l,n}$:

\begin{lem}\label{lem:condition-on-alpha}
  The following are equivalent:
  \begin{enumerate}
   \item[\upshape(a)] $\alpha_\#(X)\subseteq Y$
   \item[\upshape(b)] For every $k$-algebra surjection
     $\phi\colon\knX\to\Mn$, the composition $\phi\circ\alpha$ is also
     surjective.  {\upshape(}We will refer to this property as
     ``$\alpha$ preserves surjections onto $\Mn$''.{\upshape)}
  \end{enumerate}
\end{lem}

Note that (b) is true if $\alpha$ is an isomorphism (or a surjection),
and false if the image of $\alpha$ has PI-degree $<n$ (in which case
the proof shows that $\alpha_\#(X)\cap Y=\emptyset$).  One can
construct examples for which $\emptyset\neq\alpha_\#(X)\cap
Y\subsetneq \alpha_\#(X)$.  See also Example~\ref{ex:alpha-not-in-PIn}
below. 

\begin{proof}
  Recall that for $x\in X$, the representation
  $\phi_x\colon\Gmn\to\Mn$ given by $p\mapsto p(x)$ factors through a
  map $\bar\phi_x$ from $k_n[X]=\Gmn/\calI(X)$ onto $\Mn$.
  Conversely, any $k$-algebra surjection $\knX\to\Mn$ is of the form
  $\bar\phi_x$ for some $x\in X$, cf.~Remark~3.4.

  Note that the image of $\alpha$ is generated as $k$-algebra by the
  elements $f_i=\alpha(\overline{X_i})$.  Let $x\in X$.  By
  definition, $\alpha_\#(x)=(f_1(x),\ldots,f_l(x))$.  Note that
  $f_i(x)=(\bar\phi_x\circ\alpha)(\overline{X_i})$.  Now
  $\alpha_\#(x)\in Y$ iff $\alpha_\#(x)\in U_{l,n}$ iff $f_1(x)$,
  \ldots, $f_l(x)$ generate $\Mn$ iff $\bar\phi_x\circ\alpha$ is onto.
  Hence $\alpha_\#(X)\subseteq Y$ iff $\phi\circ\alpha$ is onto for
  all $k$-algebra surjections $\phi\colon\knX\to\Mn$.
\end{proof}

Consequently, Definition~6.1(d) needs to be modified to require that
$\alpha$ preserves surjections onto~$\Mn$ (parts (a)---(c) go through
without change):

\begin{defn}
  Let $X \subset U_{m, n}$ and $Y \subset U_{l, n}$ be $n$-varieties,
  and let $\alpha\colon \knY\to \knX$ be a $k$-algebra homomorphism
  such that for all surjective $k$-algebra maps
  $\phi\colon\knX\to\Mn$, also $\phi\circ\alpha$ is surjective.  Then
  $\alpha_\#(X)\subseteq Y$, and $\alpha_*\colon X\to Y$ is the
  regular map of $n$-varieties given by $\alpha_*(x)=\alpha_\#(x)$.
  It is easy to check that for every $g\in\knY$,
  $\alpha(g)=g\circ\alpha_*\colon X\to Y\to\Mn$.
\end{defn}

Also Remark~6.2 and Lemma~6.3 need to be modified to require that
$\alpha$ and $\beta$ preserve surjections onto $\Mn$.  With these
modifications, the proof of Theorem~1.1 in \cite{nvar} proves
Theorem~\ref{thm1prime}.  For the convenience of the reader we present
the details.

\begin{remark} \label{rem.**}
  Let $f\colon X\to Y$ be a regular map of $n$-varieties. Then for
  $x\in X$, $(f^*)_\#(x)=f(x)$, so that $(f^*)_\#(X)=f(X)\subseteq Y$.
  Hence the $k$-algebra homomorphism $f^*\colon \knY\to\knX$ preserves
  surjections onto $\Mn$.  Thus $(f^*)_*$ is defined, and $(f^*)_*=f$.
  Moreover, if $\alpha\colon \knY\to\knX$ is a $k$-algebra
  homomorphism that preserves surjections onto $\Mn$, then
  $(\alpha_*)^*=\alpha$.  Note also that $(\id_X)^*=\id_{\knX}$ and
  $(\id_{\knX})_*=\id_X$.
\end{remark}

\begin{lem} \label{lem:6.3new}
  Let $X$, $Y$ and $Z$ be $n$-varieties.
  \begin{enumerate}
  \item[\rm(a)]If $f \colon X \to Y$ and $g \colon Y \to Z$ are
    regular maps of $n$-varieties, then $(g \circ f)^* = f^* \circ g^*$.
  \item[\rm(b)]If $\alpha \colon k_n[Y] \to k_n[X]$ and $\beta \colon
    k_n[Z] \to k_n[Y]$ are $k$-algebra homomorphisms that preserve
    surjections onto $\Mn$, then also $\alpha\circ\beta$ preserves
    surjections onto $\Mn$, and $(\alpha \circ \beta)_* = \beta_*
    \circ \alpha_*$.
  \item[\rm(c)]$X$ and $Y$ are isomorphic as $n$-varieties if and only
    if $k_n[X]$ and $k_n[Y]$ are isomorphic as $k$-algebras.
  \end{enumerate}
\end{lem}

\begin{proof}
  (a) and (b) follow directly from the definitions.  The proof of
  part~(c) is the same as for Lemma~6.3(c).
\end{proof}

{\bf Theorem 6.4} goes through without changes, in arbitrary
characteristic.  It states that every finitely generated prime
$k$-algebra of PI-degree~$n$ is isomorphic to $\knX$ for some
irreducible $n$-variety~$X$ (which is unique up to isomorphism of
$n$-varieties by Lemma~\ref{lem:6.3new}(c)).

\begin{proof}[Proof of Theorem~\ref{thm1prime}] %
  By Lemma~\ref{lem:6.3new}, the functor~$\mathcal F$ in
  Theorem~\ref{thm1prime} is contravariant.  It is full and faithful
  by Remark~\ref{rem.**}.  Moreover, by Theorem~6.4, every object in
  $\PI_n$ is isomorphic to the image of an object in $\Var_n$.  Hence
  $\mathcal F$ is a contravariant equivalence of categories between
  $\Var_n$ and $\PI_n$.
\end{proof}

{\bf Lemma 6.5} goes through without changes, in arbitrary
characteristic.

We conclude this section with a few remarks about the category
$\PI_n$. 

\begin{example}\label{ex:alpha-not-in-PIn}
  Let $t$ be an indeterminate and let $S=\M_2(k[t])$.  Let $R$ be the
  subalgebra of $S$ consisting of all matrices whose lower left entry
  is contained in $tk[t]$.  Then $R$ and $S$ are prime PI-algebras in
  $\PI_2$, but the inclusion map $\beta\colon R\to S$ is not a
  morphism in $\PI_2$: $\varphi\circ\beta$ is not onto for the
  surjection $\varphi\colon S\to \M_2$ defined by $t\mapsto 0$.

  Note that by Theorem~6.4, there are, for $n=2$, irreducible
  $n$-varieties $X$ and $Y$ such that $R\cong\knY$ and $S\cong \knX$.
  Then the $k$-algebra homomorphism $\alpha\colon\knY\to\knX$ induced
  by $\beta$ does not satisfy the two equivalent conditions of
  Lemma~\ref{lem:condition-on-alpha}.
\end{example}

\begin{remark}\label{rem:alpha-in-PIn}
  Here is an alternate description of the morphisms in the category
  $\PI_n$ (and of the $k$-algebra homomorphisms~$\alpha$ satisfying
  the equivalent conditions of Lemma~\ref{lem:condition-on-alpha}).
  For a PI-algebra~$R$ in $\PI_n$, denote by $\Max_n(R)$ the maximal
  ideals of PI-degree~$n$ of $R$, i.e., the ideals $I$ such that
  $R/I\simeq\Mn$.  Now let $R$ and $S$ be PI-algebras in $\PI_n$, and
  let $\alpha\colon R\to S$ be an arbitrary $k$-algebra homomorphism.
  Then $\alpha$ is a morphism in $\PI_n$ iff $\alpha$ induces a map
  $\Max_n(S)\to\Max_n(R)$, i.e., iff for every $I\in\Max_n(S)$,
  $\alpha\inv(I)\in\Max_n(R)$.  Note that in general, $\alpha\inv(I)$
  need not even be a prime ideal of $R$ (e.g., for $I=\M_2(tk[t])$ in
  Example~\ref{ex:alpha-not-in-PIn}).
\end{remark}

\section{The First Extension of Definition~7.5}
\label{section:first-extension-S7}

It is not known if Proposition~7.3 is true in prime characteristic.
Because of this, there are two (possibly nonequivalent) ways to extend
Definition~7.5 to prime characteristic, which are explored in this and
the following section.  The first yields a category equivalence
resembling Theorem~1.2 (Theorem~\ref{thm:B}), the second only a full
and faithful functor (Theorem~\ref{thm:Dnew}).

{\bf Definition 7.1} goes through in prime characteristic:  For an
irreducible $n$-variety~$X$,  $k_n(X)$ is defined to be the total ring
of fractions of $\knX$.  It is a central simple algebra of degree~$n$,
called the {\em central simple algebra of rational functions on~$X$}.

{\bf Remark 7.2} and the subsequent discussion go through in prime
characteristic.

{\bf Proposition 7.3} may or may not go through (cf.\ 
Proposition~\ref{prop:A} below).

{\bf Definition 7.5} goes through in prime characteristic if one
strikes the second sentence of part~(a); the other parts go through
unchanged.  Because of its importance, we restate it with this one
change:

\begin{defn}\label{defn:7.5new}
  Let $X \subset U_{m, n}$ and $Y \subset U_{l, n}$ be irreducible
  $n$-varieties.
  \begin{enumerate}
  \item[(a)]A rational map $f \colon X \dasharrow Y$ is called a {\em
      rational map of $n$-varieties} if $f = (f_1,\ldots,f_l)$ where
    each $f_i\in k_n(X)$.

  \item[(b)]The $n$-varieties $X$ and $Y$ are called {\em birationally
      isomorphic} or {\em birationally equivalent} if there exist
    dominant rational maps of $n$-varieties $f \colon X \dasharrow Y$
    and $g \colon Y \dasharrow X$ such that $f \circ g = \id_Y$ and $g
    \circ f = \id_X$ (as rational maps of varieties).

  \item[(c)]A dominant rational map $f=(f_1,\ldots,f_l) \colon X
    \dasharrow Y$ of $n$-varieties induces a $k$-algebra homomorphism
    (i.e., an embedding) $f^* \colon k_n(Y) \lra k_n(X)$ of central
    simple algebras defined by $f^*(\overline{X_i})=f_i$, where
    $\overline{X_i}$ is the image of the generic matrix $X_i\in
    G_{l,n}$ in $k_n[Y] \subset k_n(Y)$.  One easily verifies that for
    every $g\in k_n(Y)$, $f^*(g)=g\circ f$, if one views $g$ as a
    $\PGLn$-equivariant rational map $Y\dasharrow\Mn$.

  \item[(d)]Conversely, a $k$-algebra homomorphism (necessarily an
    embedding) of central simple algebras $\alpha \colon k_n(Y) \lra
    k_n(X)$ (over $k$) induces a dominant rational map $f = \alpha_*
    \colon X \dasharrow Y$ of $n$-varieties.  This map is given by $f
    = (f_1, \dots, f_l)$ with $f_i = \alpha(\overline{X_i}) \in
    k_n(X)$, where $\overline{X_1}, \dots, \overline{X_l}$ are the
    images of the generic matrices $X_1, \dots, X_l \in G_{l, n}$.  It
    is easy to check that for every $g\in k_n(Y)$,
    $\alpha(g)=g\circ\alpha_*$, if one views $g$ as a
    $\PGLn$-equivariant rational map $Y\dasharrow\Mn$.
  \end{enumerate}
\end{defn}

Note that a rational map of $n$-varieties $f\colon X\dasharrow Y$ is
still a $\PGLn$-equivariant rational map of algebraic varieties, but
in prime characteristic maybe not vice versa.  The map $\alpha_*$
defined in part~(d), which is a priori a rational map $X\dasharrow
(\Mn)^l$, is indeed (in arbitrary characteristic) a dominant rational
map from $X$ to~$Y$:

\begin{lem}\label{lem:first-ext-of-def-7.5}
  Let $X \subset U_{m, n}$ and $Y \subset U_{l, n}$ be irreducible
  $n$-varieties.  Given a nonzero $k$-algebra homomorphism
  $\alpha\colon k_n(Y)\to k_n(X)$, the map $\alpha_*\colon
  X\dasharrow(\Mn)^l$ induces a dominant rational map of $n$-varieties
  $X\dasharrow Y$ {\upshape(}again denoted by $\alpha_*${\upshape)}.
\end{lem}

\begin{proof}
  Denote the generic matrices generating $G_{l,n}$ by $X_1$, \ldots,
  $X_l$, and the image of $p\in G_{l,n}$ in $\knY$ by $\bar p$.  The
  map $f=\alpha_*\colon X\dasharrow (\Mn)^l$ is defined by
  $f=(f_1,\ldots,f_l)$ where $f_i=\alpha(\overline{X_i})$.  Let $U$ be
  the $\PGLn$-invariant nonempty open subset of $X$ on which
  $f$ is defined.  Let $x\in U$.  Then for $p=p(X_1,\ldots,
  X_l)\in G_{l,n}$,
  \begin{equation}\label{lem:first-ext-of-def-7.5:eqn}
  \begin{split}
  p(f(x))
  &=p\bigl(\alpha(\overline{X_1}\,),\ldots,\alpha(\overline{X_l}\,)\bigr)(x)\\
  &=\alpha\bigl(p(\overline{X_1},\ldots,\overline{X_l}\,)\bigr)(x)
  =\alpha(\bar p)(x)\,.
  \end{split}
  \end{equation}
  In particular, $\calI(Y)$ vanishes at $f(x)$, and $\alpha(\bar p)$
  is defined on $U$.  Consequently, $f(x)\in Y$ iff $f(x)\in U_{l,n}$
  iff the elements $f_i(x)=\alpha(\overline{X_i}\,)(x)$ generate $\Mn$
  iff the $k$-algebra homomorphism $\psi_x\colon \knY\to\Mn$ defined
  by $\psi_x(\bar p)=\alpha(\bar p)(x)$ is onto.

  {\sc Claim:} {\em There is a nonempty open subset $V$ of $U$ with
    $f(V)\subseteq Y$.}  To prove this, let $t$ be a nonzero central
  element of $\knY$ such that $A=\knY[t\inv]$ is an Azumaya algebra of
  PI-degree~$n$.  Then $\alpha(t)$ is a nonzero central element of
  $k_n(X)$, so of the form $\alpha(t)=s_1s_2\inv$ for nonzero central
  elements $s_1$ and $s_2$ of $\knX$.  Denote by $V$ the nonempty open
  subset of $U$ where both $s_1$ and $s_2$ are nonzero.  Let $x\in V$.
  Then by Lemma~6.5, $\psi_x(t)=\alpha(t)(x)=s_1(x)s_2(x)\inv$ is a
  nonzero scalar in $\Mn$.  Hence $\psi_x$ extends to a $k$-algebra
  homomorphism $A\to\Mn$ which has the same image as $\psi_x$, and
  which is onto since $A$ is Azumaya of PI-degree~$n$.  Thus
  $f(V)\subseteq Y$, proving the claim.

  The claim implies that we can think of $f$ as a rational map
  $X\dasharrow Y$; we now prove that it is dominant.  Since $f$ is
  $\PGLn$-equivariant, we can replace $V$ by a $\PGLn$-invariant
  nonempty open subset of $U$ (and thus of $X$) such that
  $f(V)\subseteq Y$.  Then the closure of $f(V)$ in $Y$ is
  $\PGLn$-invariant, so an $n$-variety $Y'\subseteq Y$.  Clearly
  $\calI(Y)\subseteq\calI(Y')$.  Let $p\in\calI(Y')$ and $x\in V$.
  Then by \eqref{lem:first-ext-of-def-7.5:eqn}, $0=p(f(x))=\alpha(\bar
  p)(x)$.  Since $V$ is dense in $X$, $\alpha(\bar p)=0$, so that
  $\bar p=0$, i.e., $p\in\calI(Y)$.  Hence $\calI(Y)=\calI(Y')$, so
  that $Y=Y'$ by Theorem~5.7.  This proves that the rational map
  $f\colon X\dasharrow Y$ is dominant.
\end{proof}

{\bf Remark 7.6}, Lemma 7.7, and Theorem 7.8 all go through in
arbitrary characteristic.  Remark~7.6 states that $(f^*)_*=f$,
$(\alpha_*)^*=\alpha$, ($\id_X)^*=\id_{k_n(X)}$, and
$(\id_{k_n(X)})_*=\id_X$.  Lemma~7.7 states, in particular, that
$(g\circ f)^*=f^*\circ g^*$ and
$(\alpha\circ\beta)_*=\beta_*\circ\alpha_*$.  Theorem~7.8 contains the
fact that every central simple algebra in $\CS_n$ is isomorphic to
$k_n(X)$ for some irreducible $n$-variety $X$ (which is unique up to
birational isomorphism of $n$-varieties).

\begin{proof}[Proof of Theorem~\ref{thm:B}]
  \newcommand{\calF}{\mathcal F}%
  Call the functor $\calF$.  Since $(g\circ f)^*=f^*\circ g^*$
  (Lemma~7.7(a)), $\calF$ is a contravariant functor.  Since
  $\alpha=(\alpha_*)^*$ and $(f^*)_*=f$ (Remark~7.6), $\calF$ is full
  and faithful.  By Theorem~7.8, every object in $\CS_n$ is isomorphic
  to the image of an object under $\calF$.  Hence $\calF$ is a
  contravariant category equivalence.
\end{proof}

\section{The Second Extension of Definition~7.5}
\label{section:second-extension-S7}

We can also extend Definition~7.5 by using $\PGLn$-equivariant
rational maps of varieties instead of the above defined ``rational
maps of $n$-varieties''.  This will allow us to prove
Theorem~\ref{thm:Dnew}.  First we need to describe the basic
relationship between $k_n(X)$ and $\RMaps_\PGLn(X,\Mn)$ for an
$n$-variety~$X$; this result is a partial replacement for
Proposition~7.3.  This relationship will be further investigated in
Section~\ref{section:RMaps}.

\begin{prop}\label{prop:A}
  Let $X$ be an irreducible $n$-variety.  There is a natural embedding
  \[\psi_X\colon k_n(X) \hookrightarrow \RMaps_\PGLn(X,\Mn)\,,\]%
  allowing us to identify $k_n(X)$ with a $k$-subalgebra of
  $\RMaps_\PGLn(X,\Mn)$.  The latter algebra is a central simple
  algebra of degree~$n$, whose center we can identify with
  $k(X)^\PGLn$.  Moreover, $k(X)^\PGLn$ is a finite, purely
  inseparable extension of the center of $k_n(X)$, and
  \[\RMaps_\PGLn(X,\Mn)=k_n(X)\cdot k(X)^\PGLn\,.\]%
\end{prop}

We begin with a lemma.

\begin{lem}\label{lem-for-Prop:A}
  Let $X$ be an irreducible $\PGLn$-variety.  If the $k$-algebra
  $\RMaps_\PGLn(X,\Mn)$ contains a central simple algebra~$A$ of
  degree~$n$, then it is itself a central simple algebra of
  degree~$n$, and its center is the field $k(X)^\PGLn$.
\end{lem}

\begin{proof}
  Since $X$ is irreducible, $k(X)$ is a field, and
  $R=\RMaps_\PGLn(X,\Mn)$ embeds naturally into
  $\RMaps(X,\Mn)=\Mn(k(X))$.  Since $A$ and $\Mn(k(X))$ are central
  simple algebras of degree~$n$, the latter is a central extension of
  the former, and thus also of $R$.  Since $\Mn(k(X))$ is a prime
  ring, so is $R$.  Since $R$ contains a central simple algebra of
  degree~$n$ and is contained in a central simple algebra of
  degree~$n$, $R$ must have PI-degree~$n$.  Consequently, the center
  of $R$ is contained in the center of $\Mn(k(X))$, i.e., in
  $k(X)\cdot I_n$ (see, e.g., \cite[Lemma~4.9]{braun-vn}).

  Thus the center of $R$ consists of the $\PGLn$-equivariant rational
  maps from $X$ to $k=k\cdot I_n$.  Hence the center of $R$ can be
  identified with $k(X)^\PGLn$, the field of $\PGLn$-invariant
  rational maps $X\dasharrow k$.  Since a prime PI-algebra whose
  center is a field is simple, $R$ is a central simple algebra of
  degree~$n$.
\end{proof}

\begin{proof}[Proof of Proposition~\ref{prop:A}]
  Set $R=\RMaps_\PGLn(X,\Mn)$.  The natural embedding $\psi_X\colon
  k_n(X) \hookrightarrow R$ is constructed in the discussion after
  Remark~7.2.  Lemma~\ref{lem-for-Prop:A} implies that $R$ is a
  central simple algebra of degree~$n$ with center $k(X)^\PGLn$.  It
  remains to show that $k(X)^\PGLn$ is a finite, purely inseparable
  extension of the center $F$ of $k_n(X)$.

  By a theorem of Rosenlicht (see \cite[Theorem~13.5.3]{FSR} or
  \cite[Theorem~6.2]{dolgachev}), there is a $\PGLn$-stable dense open
  subset $X_0$ in $X$ and an algebraic variety $Y$ with
  $k(Y)=k(X)^\PGLn$ such that the inclusion map $k(Y)\hookrightarrow
  k(X)$ induces a morphism $\pi\colon X_0\to Y$ with the following
  properties: $\pi$ is open and surjective, and the fibers of $\pi$
  are the $\PGLn$-orbits in $X_0$.  Let $Y'$ be an irreducible
  algebraic variety with $k(Y')=F$.  Then the inclusion
  $F\hookrightarrow k(Y)$ includes a dominant rational map $\rho\colon
  Y\dasharrow Y'$.

  By Proposition~\ref{prop:2.3-in-prime-char}(a), the elements of
  $\Cmn$ separate the $\PGLn$-orbits in $\Umn$.  As discussed in
  Section~\ref{section:inv-background}, we can also in prime
  characteristic naturally identify $\Cmn$ with the center of $\Tmn$.
  Hence $F$, which is the center of the total ring of fractions of
  $\Tmn/\calI_T(X)$ (c.f.\ Remark~7.2), separates the $\PGLn$-orbits
  in a dense open subset of $X$.  Thus there is a dense open subset
  $U$ of $X_0$ such that the fibers of $\rho\circ\pi$ on $U$ are the
  $\PGLn$-orbits in $U$.  Hence $\rho\colon Y\dasharrow Y'$ is
  injective on the dense open subset $\pi(U)$ of $Y$.  Consequently,
  $k(Y)=k(X)^\PGLn$ is a finite, purely inseparable extension of
  $k(Y')=F$ (see, e.g., \cite[Theorem~5.1.6]{springer:LAG}).
\end{proof}

\begin{defn}\label{defn:second-definition}
  Let $X \subset U_{m, n}$ and $Y \subset U_{l, n}$ be irreducible
  $n$-varieties, so that $\RMaps_\PGLn(X,\Mn)$ and
  $\RMaps_\PGLn(Y,\Mn)$ are central simple algebras of degree~$n$ (see
  Proposition~\ref{prop:A}).
  \begin{enumerate}
  \item[(a)] Let $f\colon X\dasharrow Y$ be a $\PGLn$-quivariant
    dominant rational map (of varieties).  Define
    \[f^\star\colon \RMaps_\PGLn(Y,\Mn)\to \RMaps_\PGLn(X,\Mn)\]%
    by $f^\star(g)=g\circ f$.  Then $f^\star$ is an injective $k$-algebra
    homomorphism.
  \item[(b)]Conversely, let $\alpha\colon \RMaps_\PGLn(Y,\Mn)\to
    \RMaps_\PGLn(X,\Mn)$ be a $k$-algebra homomorphism (necessarily
    injective).  Denote the generic matrices generating $G_{l,n}$ by
    $X_1$, \ldots, $X_l$, and their images in
    $\knY\subseteq\RMaps_\PGLn(Y,\Mn)$ by $\overline{X_1}$, \ldots,
    $\overline{X_l}$.  Then $f=\alpha_\star$ is the $\PGLn$-equi\-variant
    dominant rational map $X \dasharrow Y$ defined by
    \[f=(\alpha(\overline{X_1}),\ldots,\alpha(\overline{X_l}))\,.\]%
  \end{enumerate}
\end{defn}

In this context, we denote the induced maps using a star ($\star$)
instead of an asterisk ($*$) to make it possible to distinguish the
maps defined here from the ones discussed in the previous section.
This will be useful in Section~\ref{section:RMaps}, see, e.g.,
Remarks~\ref{rem:psi-X-natural} and~\ref{rem:psi-X-natural-II}.

That $\alpha_\star$ is well-defined follows from the following, slightly
more general result (which we will use in
Section~\ref{section:RMaps}), applied to the restriction of $\alpha$
to $k_n(Y)$.

\begin{lem}\label{lem:second-definition}
  Let $X$ be an irreducible $\PGLn$-variety and $Y\subseteq U_{l,n}$
  an irreducible $n$-variety.  Let $\alpha\colon
  k_n(Y)\to\RMaps_\PGLn(X,\Mn)$ be a $k$-algebra embedding.  Define
  $f$ as in
  Definition~\ref{defn:second-definition}{\upshape(}b{\upshape)}.
  Then $f$ is a $\PGLn$-equivariant dominant rational map $X\dasharrow
  Y$.
\end{lem}

\begin{proof}
  The map $f$ is clearly a $\PGLn$-equivariant rational map
  $X\dasharrow (\Mn)^l$.  That it is actually a dominant rational map
  $X\dasharrow Y$ is shown as in the proof of
  Lemma~\ref{lem:first-ext-of-def-7.5}; the argument goes through
  nearly literally, with exception of the proof of the claim, which we
  modify as follows: By Lemma~\ref{lem-for-Prop:A},
  $\RMaps_\PGLn(X,\Mn)$ is a central simple algebra of degree~$n$ with
  center $k(X)^\PGLn$.  So the central element $\alpha(t)$ is a
  $\PGLn$-invariant rational map $X\dasharrow k\cdot I_n$.  Let $V$ be
  the nonempty open subset of $U$ where $\alpha(t)$ is defined and
  nonzero.  Then for $x\in V$, $\psi_x(t)=\alpha(t)(x)$ is again a
  nonzero scalar in $\Mn$.  The rest of the argument goes through
  unchanged.
\end{proof}

\begin{lem}\label{lem:II.8.5}
  Let $X$ and $Y$ be irreducible $n$-varieties.  Let
  \[\alpha\colon \RMaps_\PGLn(Y,\Mn)\to\RMaps_\PGLn(X,\Mn)\]%
  be a $k$-algebra homomorphism.  Then $(\alpha_\star)^\star=\alpha$.  That
  is, $\alpha(g)=g\circ\alpha_\star$ for all $g\in\RMaps_\PGLn(Y,\Mn)$.
\end{lem}

\begin{proof}
  Set $\beta=(\alpha_\star)^\star$, so that by definition,
  $\beta(g)=g\circ\alpha_\star$.  We have to show that $\beta=\alpha$.
  For $g=\bar p\in\knY$, the argument at the beginning of the proof of
  Lemma~\ref{lem:first-ext-of-def-7.5} shows that $\beta(g)=\bar
  p\circ\alpha_\star=p\circ\alpha_\star=\alpha(\bar p)=\alpha(g)$, see
  \eqref{lem:first-ext-of-def-7.5:eqn}.  Hence $\beta=\alpha$ by the
  next lemma.
\end{proof}

\begin{lem}\label{lem:8.4}
  Let $X$ and $Y$ be irreducible $n$-varieties.  Let
  \[\alpha,\beta\colon \RMaps_\PGLn(Y,\Mn)\to\RMaps_\PGLn(X,\Mn)\]%
  be $k$-algebra homomorphisms agreeing on $k_n[Y]$ {\upshape(}more
  precisley, agreeing on $\psi_Y(\knY)${\upshape)}.  Then
  $\alpha=\beta$.
\end{lem}

\begin{proof}
  In characteristic zero, this is clear by Proposition~7.3.  So assume
  the characteristic of $k$ is $p\neq0$.  By Proposition~\ref{prop:A},
  $\RMaps_\PGLn(Y,\Mn)$ is generated by $k_n(Y)$ and $k(Y)^\PGLn$.
  Since $k_n(Y)$ is the total ring of fractions of $k_n[Y]$, it
  suffices to show that $\alpha(g)=\beta(g)$ for an arbitrary $g\in
  k(Y)^\PGLn$.  By Proposition~\ref{prop:A}, there is an integer $N$
  such that $g^{p^N}$ belongs to the center of $k_n(Y)$.  Hence
  $\alpha(g^{p^N})=\beta(g^{p^N})$.  Since $\alpha$ and $\beta$ map
  central elements to central elements, $\alpha(g)$ and $\beta(g)$
  both belong to $k(X)^\PGLn$.  Let $x\in X$ such that $\alpha(g)$ and
  $\beta(g)$ are both defined at $x$.  Then $\alpha(g)(x)$ and
  $\beta(g)(x)$ are scalars in $k=k\cdot I_n$.  Now
  \[\alpha(g^{p^N})(x)=\alpha(g)^{p^N}(x)=[\alpha(g)(x)]^{p^N}\,.\]%
  Similarly, $\beta(g^{p^N})(x)=[\beta(g)(x)]^{p^N}$.  Hence
  $[\alpha(g)(x)]^{p^N}=[\beta(g)(x)]^{p^N}$.  Since the Frobenius map
  is injective, it follows that $\alpha(g)(x)=\beta(g)(x)$.  Since
  this is true for all $x$ in a dense open subset of $X$,
  $\alpha(g)=\beta(g)$.
\end{proof}

We now consider irreducible $\PGLn$-varieties that are
$\PGLn$-equivariantly birationally isomorphic to irreducible
$n$-varieties.  Let $f\colon Y\dasharrow X$ be a $\PGLn$-equivariant
dominant rational map between two such $\PGLn$-varieties.  Then
$f^\star(g)=g\circ f$ induces a $k$-algebra homomorphism
\[f^\star\colon\RMP{X}\to\RMP{Y}\,.\]
Note that $\RMP{X}$ and $\RMP{Y}$ are central simple algebra of
degree~$n$: For example, say $X$ is $\PGLn$-equivariantly birationally
isomorphic to the $n$-variety~$X'$.  Then $\RMP{X}$ is isomorphic to
$\RMP{X'}$, which is a central simple algebra of degree~$n$ by
Proposition~\ref{prop:A}.

\begin{prop}\label{prop:II.8.7}
  Let $X$ and $Y$ be irreducible $\PGLn$-varieties that are
  $\PGLn$-equivariantly birationally isomorphic to irreducible
  $n$-varieties.  Let
  \[\alpha\colon\RMP{X}\to\RMP{Y}\]
  be a $k$-algebra homomorphism.  Then there is a unique
  $\PGLn$-equivariant dominant rational map $\alpha_\star\colon
  Y\dasharrow X$ such that $(\alpha_\star)^\star=\alpha$.
\end{prop}

\begin{proof}
  Assume first that $X$ and $Y$ are $n$-varieties.  Then we can take
  $\alpha_\star$ as in Definition~\ref{defn:second-definition}, and
  $(\alpha_\star)^\star=\alpha$ by Lemma~\ref{lem:II.8.5}.  Say
  $h\colon Y\dasharrow X$ is another $\PGLn$-equivariant dominant
  rational map such that $h^\star=\alpha$.  It follows easily from
  Definition~\ref{defn:second-definition} that $(h^\star)_\star=h$.
  Hence $h=(h^\star)_\star=\alpha_\star$, proving the uniqueness of
  $\alpha_\star$.

  In the general case, say $f\colon X\dasharrow X'$ and $g\colon
  Y\dasharrow Y'$ are $\PGLn$-equivariant birational isomorphisms,
  where $X'$ and $Y'$ are irreducible $n$-varieties.  Consider the
  $k$-algebra homomorphism
  \[\beta = (g^\star)\inv \circ \alpha \circ
  f^\star \colon \RMaps_\PGLn(X',\Mn) \to \RMaps_\PGLn(Y',\Mn)\,.\]
  By the first case, there is a unique $\PGLn$-equivariant, dominant
  rational map $\beta_\star \colon Y' \dasharrow X'$ such that
  $(\beta_\star)^\star =\beta$.  One sees now as in the proof of
  Lemma~8.5 that $\alpha_\star$ exists and is unique; the only
  modification needed is to replace $k_n(\underline{\hphantom{X}})$ by
  $\RMaps_\PGLn(\underline{\hphantom{X}},\Mn)$.  (Note the typo in the
  diagram in the proof of Lemma~8.5: the arrow for $\beta_*$ should be
  reversed.)
\end{proof}

\begin{cor}\label{cor:II.8.8}
  Let $X$ and $Y$ be irreducible $\PGLn$-varieties that are
  $\PGLn$-equivariantly birationally isomorphic to irreducible
  $n$-varieties.  Let $f\colon X\dasharrow Y$ be a $\PGLn$-equivariant
  dominant rational map.  Then $(f^\star)_\star=f$.
\end{cor}

\begin{proof}
  Set $\alpha=f^\star$.  Since $\alpha_\star$ is the unique such map
  with the property that $(\alpha_\star)^\star=\alpha$,
  $f=\alpha_\star=(f^\star)_\star$.
\end{proof}

\begin{cor}\label{cor:II.8.9}
  Let $X$, $Y$ and $Z$ be irreducible $\PGLn$-varieties that are
  $\PGLn$-equivariantly birationally isomorphic to irreducible
  $n$-varieties.
  \begin{enumerate}
  \item[\rm(a)] If $f \colon X \dasharrow Y$ and $g \colon Y \dasharrow
    Z$ are $\PGLn$-equivariant dominant rational maps then $(g \circ
    f)^\star = f^\star \circ g^\star$.
  \item[\rm(b)]Set $R=\RMaps_\PGLn(X,\Mn)$.  Then $(\id_X)^\star = \id_R$,
    and $(\id_R)_\star = \id_X$.
  \item[\rm(c)]If $\alpha \colon \RMaps_\PGLn(Y,\Mn) \to
    \RMaps_\PGLn(X,\Mn)$ and\\ $\beta \colon \RMaps_\PGLn(Z,\Mn) \to
    \RMaps_\PGLn(Y,\Mn)$ are $k$-algebara homomorphisms then $(\alpha
    \circ \beta)_\star = \beta_\star \circ \alpha_\star$.
  \item[\rm(d)]$X$ and $Y$ are birationally isomorphic as
    $\PGLn$-varieties if and only if $\RMaps_\PGLn(X,\Mn)$ and
    $\RMaps_\PGLn(Y,\Mn)$ are isomorphic as $k$-algebras.
  \end{enumerate}
\end{cor}

\begin{proof} (a) and the first identity in (b) are immediate from
  Definition~\ref{defn:second-definition}.  The second identity in~(b)
  follows from the first and Corollary~\ref{cor:II.8.8}. (c)~Let $f =
  (\alpha \circ \beta)_\star$ and $g = \beta_\star \circ
  \alpha_\star$.  By Proposition~\ref{prop:II.8.7} and part~(a),
  $f^\star =\alpha\circ\beta = g^\star$. The uniqueness assertion of
  Proposition~\ref{prop:II.8.7} now implies $f = g$.  (d)~follows from
  (a), (b) and (c) (cf.~the proof of Lemma~6.3).
\end{proof}

\begin{proof}[Proof of Theorem~\ref{thm:Dnew}]
  By the discussion before Proposition~\ref{prop:II.8.7}, the
  assignment in Theorem~\ref{thm:Dnew} is well-defined.  It is a
  contravariant functor by Corollary~\ref{cor:II.8.9}.  It is full
  since $\alpha=(\alpha_\star)^\star$ (Proposition~\ref{prop:II.8.7}),
  and it is faithful since $(f^\star)_\star=f$
  (Corollary~\ref{cor:II.8.8}).
\end{proof}

\section{More on $\RMaps_\PGLn(X,\Mn)$}
\label{section:RMaps}

\begin{remark}\label{rem:psi-X-natural}
  The algebra embedding $\psi_X$ in Proposition~\ref{prop:A} is
  natural in the following sense.  Let $f\colon X\dasharrow Y$ be a
  dominant rational map of irreducible $n$-varieties.  Then the
  following diagram commutes:
  \[ \xymatrix{
    k_n(Y) \ar@{->}[r]^-{\psi_Y} \ar@{->}[d]^{\,f^*} 
                           & \RMaps_\PGLn(Y,\Mn) \ar@{->}[d]^{\,f^\star} \cr 
    k_n(X) \ar@{->}[r]^-{\psi_X} &  \RMaps_\PGLn(X,\Mn)} \]
  Here the induced $k$-algebra embeddings $f^*$ and $f^\star$ were
  studied in Sections~\ref{section:first-extension-S7}
  and~\ref{section:second-extension-S7} (see
  Definitions~\ref{defn:7.5new}(c)
  and~\ref{defn:second-definition}(a), respectively).
\end{remark}

\begin{remark}\label{rem:psi-X-natural-II}
  Again, let $X$ and $Y$ be irreducible $n$-varieties, and let
  $\alpha\colon k_n(Y)\to k_n(X)$ be a $k$-algebra homomorphism.  Then
  there is a unique $k$-algebra homomorphism~$\beta$ so that the
  following diagram commutes:
  \[ \xymatrix{
    k_n(Y) \ar@{->}[r]^-{\psi_Y} \ar@{->}[d]^{\,\alpha} 
                           & \RMaps_\PGLn(Y,\Mn) \ar@{..>}[d]^{\,\beta}  \cr 
    k_n(X) \ar@{->}[r]^-{\psi_X} &  \RMaps_\PGLn(X,\Mn)} \]
  Indeed, let $\beta=f^\star$
  (Definition~\ref{defn:second-definition}(a)), where $f=\alpha_*$
  (Definition~\ref{defn:7.5new}(d), see also
  Lemma~\ref{lem:first-ext-of-def-7.5}).  Then the diagram commutes by
  the previous remark since $\alpha=f^*$ (Remark~7.6).
  Lemma~\ref{lem:8.4} shows that $\beta$ is unique.
\end{remark}

\begin{lem}\label{lem:9.1}
  Let $X$ be an irreducible $\PGLn$-variety and $Y$ an irreducible
  $n$-variety.  Assume that there is a $\PGLn$-equivariant dominant
  rational map $f\colon X\dasharrow Y$.  Then $g\mapsto g\circ f$
  defines a $k$-algebra homomorphism
  \[f^\star\colon \RMaps_\PGLn(Y,\Mn)\lra \RMaps_\PGLn(X,\Mn)\,.\]
  Moreover, $\RMaps_\PGLn(X,\Mn)$ is a central simple algebra of
  degree~$n$ with center $k(X)^\PGLn$.
\end{lem}

\begin{proof}
  It is clear that $f^\star$ is a $k$-algebra homomorphism.  Since the
  algebra $\RMaps_\PGLn(Y,\Mn)$ is a central simple algebra of
  degree~$n$ by Proposition~\ref{prop:A}, so is $\RMaps_\PGLn(X,\Mn)$
  by Lemma~\ref{lem-for-Prop:A}.
\end{proof}

\begin{lem}\label{lem:9.2-new}
  In the situation of Lemma~\ref{lem:second-definition}:
  \begin{enumerate}
  \item[\upshape(a)]Let
    $f^\star\colon\RMaps_\PGLn(Y,\Mn)\to\RMaps_\PGLn(X,\Mn)$ be the map
    induced by $f$ {\upshape(}via $g\mapsto g\circ f${\upshape)}.
    Then $\alpha=f^\star\circ\psi_Y$, i.e., the following diagram
    commutes:
  \[\xymatrix{
    k_n(Y) \ar@{->}[r]^-{\ \psi_Y\ } \ar@{->}[rd]_(0.42)\alpha
                           & \RMaps_\PGLn(Y,\Mn) \ar@{->}[d]^(0.5){\,f^\star} \cr 
                           &  \RMaps_\PGLn(X,\Mn)}\]
  \item[\upshape(b)]If $\alpha$ is an isomorphism, so are $\psi_Y$ and
    $f^\star$. 
  \item[\upshape(c)]Assume that also $X$ is an $n$-variety.  If
    $\psi_X$ is an isomorphism {\upshape(}or if
    $\alpha(k_n(Y))\subseteq \psi_X(k_n(X))${\upshape)}, then $f$ is a
    dominant rational map of $n$-varieties.
  \end{enumerate}
\end{lem}

\begin{proof}
  (a) Since $k_n(Y)$ is the total ring of fractions of $\knY$, it
  suffices to prove that for $p\in G_{l,n}$, $(f^\star\circ\psi_Y)(\bar
  p)=\alpha(\bar p)$.  This is true since for $x\in X$ in general
  position,
  \[[f^\star(\psi_Y(\bar p))](x)=[\psi_Y(\bar p)](f(x))
  =p(f(x))=\alpha(\bar p)(x)\,,\]
  where the last equality follows from
  \eqref{lem:first-ext-of-def-7.5:eqn}.  (b)~follows from (a) since
  $\psi_Y$ and $f^\star$ are embeddings.  (c)~is clear from the
  definition of~$f$ in Lemma~\ref{lem:second-definition} (cf.\
  Definitions~\ref{defn:7.5new}(d)
  and~\ref{defn:second-definition}(b)).
\end{proof}

\begin{prop}\label{prop:9.3}
  Let $X$ be an irreducible $\PGLn$-variety such that the algebra
  $\RMaps_\PGLn(X,\Mn)$ is a central simple algebra of degree~$n$.
  Then there is an irreducible $n$-variety $Y$ with an isomorphism
  \[\alpha\colon k_n(Y)\to\RMaps_\PGLn(X,\Mn)\,.\]
  For any such $Y$ and $\alpha$, the $\PGLn$-equivariant dominant
  rational map $f\colon$\linebreak[0]$X\dasharrow Y$ from
  Lemma~\ref{lem:second-definition} is generically injective, and both
  $\psi_Y$ and $f^\star$ are isomorphisms.  In particular,
  \[k_n(X)\xrightarrow[\mbox{\rm\tiny 1-1}]{\psi_X}\RMaps_\PGLn(X,\Mn)
  \xrightarrow[\cong]{\alpha\inv} k_n(Y)
  \xrightarrow[\cong]{\psi_Y}\RMaps_\PGLn(Y,\Mn)\,.\] 
\end{prop}

\begin{proof}
  By Theorem~7.8, there is an irreducible $n$-variety $Y$ with a
  $k$-algebra isomorphism $\alpha\colon k_n(Y)\to
  \RMaps_\PGLn(X,\Mn)$.  Let $f\colon X\dasharrow Y$ be the
  $\PGLn$-equivariant dominant rational map from
  Lemma~\ref{lem:second-definition}.  Then $\psi_Y$ and $f^\star$ are
  isomorphisms by Lemma~\ref{lem:9.2-new}(b).  It remains to be shown
  that $f$ is generically injective (which we will use only once in
  the sequel, namely in Corollary~\ref{cor:9.3}).

  Since $f^\star\colon\RMaps_\PGLn(Y,\Mn)\to\RMaps_\PGLn(X,\Mn)$ is an
  isomorphism, it maps the center isomorphically onto the center,
  i.e., $f^\star$ maps $k(Y)^\PGLn$ isomorphically onto $k(X)^\PGLn$.
  Denote by $\pi_X\colon X\dasharrow \overline{X}$ and $\pi_Y\colon
  Y\dasharrow\overline{Y}$ the rational quotients of $X$ and $Y$ with
  respect to the $\PGLn$-actions (see \cite[Section 6.3]{dolgachev}).
  Then we obtain the following commutative diagram:
  \[ \xymatrix{
     X \ar@{-->}[r]^{f} \ar@{-->}[d]_-{\pi_X}  & Y \ar@{-->}[d]^-{\pi_Y} \cr 
    \overline{X} \ar@{-->}[r]^{\bar f}       & \overline{Y} } \]
  Here $\bar f$ is induced by the isomorphism from
  $k(\overline{Y})=k(Y)^\PGLn$ onto $k(\overline{X})=k(X)^\PGLn$, so
  is a birational isomorphism.  Now let $x_1,x_2\in X$ be in general
  position.  If $f(x_1)=f(x_2)$, then $\pi_X(x_1)=\pi_X(x_2)$ since
  $\bar f$ is a birational isomorphism.  Hence $x_1$ and $x_2$ belong
  to the same $\PGLn$-orbit.  So $x_2=hx_1$ for some $h\in\PGLn$.
  Then $f(x_1)=f(x_2)=hf(x_1)$.  This can be true only if $h=1$, i.e.,
  if $x_1=x_2$, since every point in the $n$-variety $Y$ has trivial
  stabilizer in $\PGLn$, see \eqref{eq:n-var-triv-stab}.
\end{proof}

\begin{cor}\label{cor:9.3}
  Let $X$ be an irreducible $\PGLn$-variety.  The following are
  equivalent: 
  \begin{enumerate}
  \item[\upshape(a)]$\RMaps_\PGLn(X,\Mn)$ is a central simple algebra
    of degree~$n$.
  \item[\upshape(b)]There is a $\PGLn$-equivariant dominant rational
    map $f\colon X\dasharrow Y$ for some irreducible $n$-variety $Y$.
  \end{enumerate}
  If so, the $\PGLn$-action on $X$ has trivial stabilizers at points
  in general position, and $Y$ and $f$ can be chosen such that $f$ is
  in addition generically injective.
\end{cor}

\begin{proof}
  This follows from Proposition~\ref{prop:9.3} and
  Lemma~\ref{lem:9.1}.  Recall from \eqref{eq:n-var-triv-stab} that
  every point in $Y$ has trivial stabilizer in $\PGLn$.  Since $f$ is
  $\PGLn$-equivariant, every $x\in X$ in general position (wherever
  $f$ is defined) must have trivial stabilizer in $\PGLn$ as well.
\end{proof}

\begin{prop}\label{prop:9.5}
  Let $X$ be an irreducible $n$-variety.  Then there is an irreducible
  $n$-variety $Y$ such that $\RMaps_\PGLn(X,\Mn)\cong k_n(Y)$.  For
  any such $Y$, there is a $\PGLn$-equivariant dominant rational map
  $f\colon X\dasharrow Y$ and a dominant rational map of $n$-varieties
  $g\colon Y\dasharrow X$ which are inverse to each other as
  rational maps.  Consequently, $X$ and $Y$ are birationally
  isomorphic as $\PGLn$-varieties {\upshape(}but maybe not as
  $n$-varieties{\upshape)}.  In addition, if also $\psi_X$ is an
  isomorphism, then $X$ and $Y$ are birationally isomorphic as
  $n$-varieties. 
\end{prop}

\begin{proof}
  By Proposition~\ref{prop:A}, $\RMaps_\PGLn(X,\Mn)$ is a central
  simple algebra of degree~$n$.  By Theorem~7.8, there is an
  irreducible $n$-variety $Y$ with a $k$-algebra isomorphism
  $\alpha\colon k_n(Y)\to \RMaps_\PGLn(X,\Mn)$.  Let $f\colon
  X\dasharrow Y$ be the $\PGLn$-equivariant dominant rational map from
  Lemma~\ref{lem:second-definition}.  Then Lemma \ref{lem:9.2-new}
  implies that $\alpha=f^\star\circ\psi_Y$, and that $\psi_Y$ and
  $f^\star$ are isomorphisms (since $\alpha$ is).  Moreover, if also
  $\psi_X$ is an isomorphism, then $f$ is a dominant rational map of
  $n$-varieties.

  Now consider the $k$-algebra embedding
  \[
  \beta=\psi_Y\circ\alpha\inv\circ\psi_X\colon
  k_n(X)\to\RMaps_\PGLn(Y,\Mn)\,.
  \]
  Applying similar reasoning to $\beta$ (and with the roles of $X$ and
  $Y$ reversed), we obtain a $\PGLn$-equivariant dominant rational map
  $g\colon Y\dasharrow X$ such that $\beta=g^\star\circ\psi_X$ (see
  Figure~\ref{figure1}).  Moreover, $g$ is a dominant rational map of
  $n$-varieties since $\psi_Y$ is an isomorphism.  Now
  \begin{figure}
  \[
    \xymatrix{
    k_n(Y) \ar@{->}[r]^-{\ \psi_Y\ } \ar@{->}[rd]^(0.4)\alpha
                    & \RMaps_\PGLn(Y,\Mn) \ar@{->}@<-1ex>[d]_(.5){f^\star} \cr 
    k_n(X) \ar@{->}[r]_-{\ \psi_X\ } \ar@{->}[ru]_(0.4)\beta
                    &  \RMaps_\PGLn(X,\Mn)  \ar@{->}@<-1ex>[u]_(.5){\,g^\star} }
  \]
  \vspace{-.3cm}
  \caption{\label{figure1}\em The maps in the proof of
    Proposition~\ref{prop:9.5}.}

  \bigskip
  \end{figure}
  \[(f^\star\circ g^\star)\circ\psi_X=f^\star\circ\beta
  =(f^\star\circ\psi_Y)\circ\alpha\inv\circ\psi_X=\id\circ\,\psi_X,\]
  where $\id$ is the identity map on $\RMaps_\PGLn(X,\Mn)$.  It
  follows from Lemma \ref{lem:8.4} that $f^\star\circ g^\star=\id$.
  Since $f^\star$ is an isomorphism, $f^\star$ and $g^\star$ are
  inverse isomorphisms.  It follows now from
  Corollary~\ref{cor:II.8.9} that $f$ and $g$ are inverse to each
  other as rational maps.
\end{proof}

\begin{cor}\label{cor:9.7}
  Let $X$ be an irreducible $n$-variety.  Then there is an irreducible
  $n$-variety $Y$, unique up to birational isomorphism of
  $n$-varieties, such that $\RMaps_\PGLn(X,\Mn)\cong k_n(Y)$.
  Moreover, $X$ and $Y$ are birationally isomorphic as
  $\PGLn$-varieties {\upshape(}but maybe not as
  $n$-varieties{\upshape)}.
\end{cor}

\begin{proof}
  All but the uniqueness of $Y$ follows directly from
  Proposition~\ref{prop:9.5}.  Let $X'$ be another irreducible
  $n$-variety such that $k_n(X')$ is isomorphic to
  $\RMaps_\PGLn(X,\Mn)$.  Then by Proposition~\ref{prop:9.3},
  $\psi_{X'}$ is an isomorphism. Applying Proposition~\ref{prop:9.5}
  with $X'$ instead of $X$, we conclude that $X'$ and $Y$ are
  birationally isomorphic as $n$-varieties.
\end{proof}

\begin{remark}
  Suppose that $f\colon X\dasharrow Y$ is a birational isomorphism of
  irreducible $n$-varieties.  Then the two vertical maps $f^*$ and
  $f^\star$ in Remark~\ref{rem:psi-X-natural} must be isomorphisms
  (since the assignments $f\mapsto f^*$ and $f\mapsto f^\star$ are
  functorial), so that $\psi_X$ must be an isomorphism if $\psi_Y$ is.
  Proposition~\ref{prop:9.5} implies therefore the following: If there
  is an irreducible $n$-variety $X$ such that the embedding
  $\psi_X\colon k_n(X)\to \RMaps_\PGLn(X,\Mn)$ is {\em not} an
  isomorphism, then there are irreducible $n$-varieties which are
  birationally isomorphic as $\PGLn$-varieties but {\em not} as
  $n$-varieties.
\end{remark}

\end{document}